\documentclass[11pt]{article}
\usepackage{amssymb, amsfonts, amsthm, mathrsfs, amsmath}
\usepackage[all,cmtip]{xy}
\usepackage{tipa}

\theoremstyle{definition}
\newtheorem{thm}{Theorem}[section]
\newtheorem{dfn}[thm]{Definition}
\newtheorem{lemma}[thm]{Lemma}
\newtheorem{cor}[thm]{Corollary}
\newtheorem{prop}[thm]{Proposition}
\newtheorem{ex}[thm]{Example}
\newtheorem{rmk}[thm]{Remark}

\def\ann{\mathrm{Ann}}
\def\endo{\mathrm{End}}

\def\img{\mathrm{Im}}
\def\id{\mathrm{id}}

\def\tr{\mathrm{Tr}}

\def\ccc{\mathbb{C}}

\def\zz{\mathbb{Z}}
\def\rr{\mathbb{R}}

\def\pp{\mathbb{P}}

\def\cle{\mathcal{E}}
\def\clh{\mathcal{H}}
\def\cll{\mathcal{L}}

\def\clo{\mathcal{O}}

\def\pt{\partial}
\def\p{\partial}

\def\ud{\mathrm{d}}

\def\vol{\mathrm{vol}}

\def\spann{\mathrm{Span}}

\def\bea{\begin{eqnarray}}
\def\eea{\end{eqnarray}}

\def\<{{\langle}}
\def\>{{\rangle}}

\begin{document}

\title{The degeneration of Calabi-Yau 3-folds via 3-forms and the SYZ conjecture}
\author{Teng Fei}

\date{}

\maketitle{}

\begin{abstract}
In this paper, we investigate the geometries associated with 3-forms of various orbital types on a symplectic 6-manifold. We demonstrate that certain unstable 3-forms, which naturally emerge from specific degenerations of Calabi-Yau structures, exhibit remarkably rich geometric properties. This, in turn, offers a novel perspective on the SYZ conjecture.
\end{abstract}

\section{Introduction}

The classification of 3-forms in real dimension six has been well-established \cite{reichel1907, gurevich1935b, gurevich1964, chan1998, bryant2006b}. A notable feature in this dimension is the existence of two open (also known as stable) orbits under the natural action of the general linear group. This fact has been utilized by many authors, including Hitchin \cite{hitchin2000} and Donaldson-Lehmann \cite{donaldson2024,donaldson2024b,donaldson2024c}, to investigate geometric and deformation problems in complex and Calabi-Yau geometries.

In this paper, we systematically investigate various geometric structures associated with 3-forms on a 6-dimensional symplectic manifold, including those from unstable orbits under the general linear group action. We introduce several natural integrability conditions on 3-forms and explore the relationships among them. In particular, we focus on the geometry of certain unstable 3-forms that naturally arise from the degeneration of Calabi-Yau structures. We demonstrate that such geometries possess a canonical Lagrangian foliation, which emerges as a consequence of the degeneration of complex structures. Moreover, we prove that the leaves of this foliation naturally admit a pair of dual Hessian structures, where the local potentials of the metric automatically satisfy a real Monge-Amp\`ere equation. When the Lagrangian foliation is a Lagrangian torus fibration, we further show that the fibration is automatically semi-flat, and the base of the fibration is naturally equipped with an $L^2$-metric and a pair of dual Hessian structures. Additionally, we provide explicit examples of this type of Calabi-Yau 3-fold degeneration.

This paper is organized as follows. In Section 2, we review the classification of 3-forms in dimension six under both the general linear group and the symplectic group actions. We introduce homogeneous equivariant polynomials to describe the orbits of 3-forms. Section 3 is dedicated to integrability conditions on 3-forms. In particular, we introduce the concept of $F$-harmonic 3-forms as an ideal integrability condition. In Section 4, we introduce the notion of wonderful degenerations of Calabi-Yau structures as a convenient set-up for our discussion. Section 5 is devoted to investigation of the geometry associated with an $F$-harmonic 3-form in the orbit $\clo_0^+$. Such 3-forms naturally arise from certain degenerations of Calabi-Yau structures and induce remarkably rich geometric structures on the underlying symplectic manifold. These include a canonical Lagrangian foliation, where the leaves are equipped with a pair of dual Hessian structures and a parallel volume form. Furthermore, when the Lagrangian foliation is a torus fibration, we show that almost all the geometric predictions of the SYZ conjecture are automatically satisfied. Finally in Section 6, we present a non-trivial example of wonderful degeneration and its associated SYZ fibration.\\

\noindent \textbf {Acknowledgment: }The author would like to thank Robert Bryant, Tristan Collins, Yu-Shen Lin, Duong H. Phong, Li-Sheng Tseng and Xiangwen Zhang for valuable discussions. This work is generously supported by the Simons Collaboration Grant 830816 and the NSF LEAPS Grant 2418918.  

\section{Classification of three-forms}

In this section, we establish a collection of basic facts about the classification of 3-forms on a 6-dimensional real vector space. A few of the statements presented here are probably new. However, due to their elementary nature, we leave the proof as exercises for the interested readers.\\

Let $V$ be a real vector space of dimension $6$. We denote by $\bigwedge^3V^*$ the space of 3-forms on $V$. The group $\mathrm{GL}(V)$ naturally acts on $\bigwedge^3V^*$. This action is ``stable'' in the sense that it has open orbits. In fact, all the orbits of this action have been classified classically, see for example \cite{reichel1907, gurevich1935b, gurevich1964, chan1998, bryant2006b}.

To better describe these orbits, we introduce the following definitions.
\begin{dfn}
A 3-form $\varphi\in\bigwedge^3V^*$ is called \emph{stable} if it lies in an open orbit under the $\mathrm{GL}(V)$-action.
\end{dfn}
For any 3-form $\varphi\in\bigwedge^3V^*$, its kernel $\ker\varphi$ is defined to be
\[\ker\varphi=\{v\in V:\iota_v\varphi=0\}.\]
\begin{dfn}
A 3-form $\varphi$ is called \emph{non-degenerate} if $\ker\varphi=\{0\}$.
\end{dfn}
As $\varphi$ induces a non-degenerate 3-form $\bar\varphi\in\bigwedge^3(V/\ker\varphi)^*$, it is easy to see that the only possible dimensions for $\ker\varphi$ are $0$, $1$, $3$, and $6$.

For later use, we introduce several equivariant homogeneous polynomials on the space $\bigwedge^3V^*$:
\begin{enumerate}
\item $K$ is an equivariant homogeneous degree 2 polynomial of $\varphi$ such that $K(\varphi)\in\endo~V\otimes\bigwedge^6V^*$ for any $\varphi$. It is  defined by 
    \[K(\varphi)(v)=-\iota_v\varphi\wedge\varphi\in\bigwedge\nolimits^5V^*\cong V\otimes\bigwedge\nolimits^6V^*\] for any $v\in V$.
\item $F$ is an equivariant homogeneous degree 3 polynomial of $\varphi$ such that $F(\varphi)\in\bigwedge^3V^*\otimes\bigwedge^6V^*$ for any $\varphi$. It is defined by
    \[F(\varphi)(v_1,v_2,v_3)=-2\varphi(K(\varphi)(v_1),v_2,v_3)\] for any $v_1,v_2,v_3\in V$. Modulo tensoring the one dimensional space $\bigwedge^6V^*$, $F(\varphi)$ is also a 3-form. Therefore it makes sense to talk about $\ker F(\varphi)$. By this definition, we know that
    \[\ker F(\varphi)=K(\varphi)^{-1}(\ker\varphi\otimes\bigwedge\nolimits^6V^*).\]
\item $Q$ is an equivariant homogeneous degree 4 polynomial of $\varphi$ such that $Q(\varphi)\in(\bigwedge^6V^*)^{\otimes 2}$ for any $\varphi$. It is defined by
    \[Q(\varphi)=-\varphi\wedge F(\varphi).\]
\end{enumerate}
\begin{dfn}
An element of $(\bigwedge^6V^*)^{\otimes 2}$ is called \emph{positive} if it can be written as $\tau^2$ for some $0\neq \tau\in\bigwedge^6V^*$. An element of $(\bigwedge^6V^*)^{\otimes 2}$ is called \emph{negative} if its negative is positive.
\end{dfn}

It has been known for a long time that the $\mathrm{GL}(V)$-action on $\bigwedge^3V^*$ has 6 distinct orbits. Two of them are stable, which are exactly the open sets defined by $Q(\varphi)<0$ and $Q(\varphi)>0$ respectively. These two open orbits will be denoted by $\clo_-$ and $\clo_+$ throughout this paper. The union of the other 4 orbits is the hypersurface $Q(\varphi)=0$. In fact, these 4 orbits can be distinguished by the dimension of their kernels. As a consequence, we shall denote these orbits by $\clo_0$, $\clo_1$, $\clo_3$, $\clo_6=\{0\}$, where the subscript labels the dimension of their kernel spaces.

Regarding these orbits, we have the following table:
\[\begin{tabular}{c|c|l}
\hline
Orbit & Dimesion & Normal Form \\
\hline
$\clo_-$ & 20 & $e^{135}-e^{146}-e^{236}-e^{245}$ \\
\hline
$\clo_+$ & 20 & $e^{123}+e^{456}$ \\
\hline
$\clo_0$ & 19 & $e^{146}+e^{236}+e^{245}$ \\
\hline
$\clo_1$ & 15 & $e^{135}+e^{245}$ \\
\hline
$\clo_3$ & 10 & $e^{135}$ \\
\hline
$\clo_6=\{0\}$ & 0 & 0 \\
\hline
\end{tabular}\]

\begin{dfn}
For $\varphi\in\bigwedge^3V^*$, the space $\ann~\varphi$ is defined to be
\[\ann~\varphi:=\{\alpha\in V^*:\alpha\wedge\varphi=0\}.\]
Dually we define
\[(\ann~\varphi)^\perp:=\{v\in V:\alpha(v)=0\textrm{ for any }\alpha\in\ann~\varphi\}.\]
\end{dfn}
Noticing that $K(\varphi)$ is a linear map from $V$ to $V\otimes\bigwedge^6V^*$ and that $V\cong V\otimes\bigwedge^6V^*$ canonically up to scaling, we may view both $\ker K(\varphi)$ and $\img~K(\varphi)$ as subspaces of $V$. Therefore associated with each $\varphi\in\bigwedge^3V^*$, there are four natural subspaces of $V$: $\ker\varphi$, $\ker K(\varphi)$, $\img~K(\varphi)$, $(\ann~\varphi)^\perp$ and one subspace of $V^*$: $\ann~\varphi$. Regarding dimensions of these vector spaces, it is easy to obtain the following table using the normal forms above.
\[\begin{tabular}{|c|c|c|c|c|}
\hline
Orbit & $\dim\ker\varphi$ & $\dim\ker K(\varphi)$ & $\dim\img~K(\varphi)$ & $\dim~(\ann~\varphi)^\perp$ \\
\hline
$\clo_-$ & 0 & 0 & 6 & 6 \\
\hline
$\clo_+$ & 0 & 0 & 6 & 6 \\
\hline
$\clo_0$ & 0 & 3 & 3 & 6 \\
\hline
$\clo_1$ & 1 & 5 & 1 & 5 \\
\hline
$\clo_3$ & 3 & 6 & 0 & 3 \\
\hline
$\clo_6=\{0\}$ & 6 & 6 & 0 & 0 \\
\hline
\end{tabular}\]
Moreover, one can check case by case that if any two spaces in the same line has the same dimension, then they are equal. Incorporating the fact that $F(\varphi)=0$ if and only if $\img~K(\varphi)\subset\ker\varphi$, we get the following stratification result:
\begin{prop}
The zeroes of the polynomials $K$, $F$, and $Q$ can be characterized as follows:
\begin{enumerate}
\item $\{\varphi:\varphi=0\}=\clo_6$,
\item $\{\varphi:K(\varphi)=0\}=\clo_3\coprod\clo_6$,
\item $\{\varphi:F(\varphi)=0\}=\clo_1\coprod\clo_3\coprod\clo_6$,
\item $\{\varphi:Q(\varphi)=0\}=\clo_0\coprod\clo_1\coprod\clo_3\coprod\clo_6$.
\end{enumerate}
\end{prop}
In addition, the polynomials $K$, $F$ and $Q$ satisfy the following equations:
\begin{prop}~
\begin{enumerate}
\item $K(\varphi)\circ K(\varphi)=\dfrac{\id_V}{4}\cdot Q(\varphi)\in\endo~V\otimes(\bigwedge^6V^*)^{\otimes2}$,
\item $K(F(\varphi))=-K(\varphi)\cdot Q(\varphi)\in\endo~V\otimes(\bigwedge^6V^*)^{\otimes3}$,
\item $F(F(\varphi))=-\varphi\cdot Q^2(\varphi)\in\bigwedge^3V^*\otimes(\bigwedge^6V^*)^{\otimes4}$.
\end{enumerate}
\end{prop}
\begin{proof}
As these identities are algebraic in nature, we only need to verify them on some open subset of $\bigwedge^3V^*$ such as $\clo_-$, where we can apply Hitchin \cite{hitchin2000}'s result below.
\end{proof}

Most work in the literature focuses on the orbit $\clo_-$ as it is related to complex and Calabi-Yau geometries. As explained in the pioneering work of Hitchin \cite{hitchin2000}, for any $\varphi\in\clo_-$, we have $K(\varphi)\circ K(\varphi)=\id_V\cdot\lambda(\varphi)$ for some $0>\lambda(\varphi)\in(\bigwedge^6V^*)^{\otimes2}$. Let $\sqrt{-\lambda(\varphi)}\in\bigwedge^6V^*$ be either of the two square roots of $-\lambda(\varphi)$, we have that
\[J(\varphi)=\frac{K(\varphi)}{\sqrt{-\lambda(\varphi)}}:V\to V\]
defines a complex structure on $V$, such that the complex valued 3-form $\varphi+\sqrt{-1}\hat\varphi$, where $\hat\varphi=J(\varphi)^*\varphi$, is a nowhere vanishing $(3,0)$-form with respect to $J(\varphi)$. In this setting, we have
\bea\label{relations}
\begin{split}K(\varphi)&=J(\varphi)\cdot\sqrt{-\lambda(\varphi)},\\
F(\varphi)&=2\hat\varphi\cdot\sqrt{-\lambda(\varphi)},\\
Q(\varphi)&=4\lambda(\varphi).\end{split}
\eea
There is a parallel story about the orbit $\clo_+$, which is a ``para-complex'' or ``split'' version of Hitchin's theory, see for example \cite{etayo2006, fei2015b, hamilton2023}.

The cubic polynomial map $F:\bigwedge^3V^*\to\bigwedge^3V^*\otimes(\bigwedge^6V^*)^{\otimes 2}$ is $\mathrm{GL}(V)$-equivariant, therefore it maps orbits to orbits. It is not hard to deduce the following statement.
\begin{prop}\label{Fmap}
The cubic polynomial $F$ satisfies the following:
\begin{enumerate}
\item $F:\clo_\pm\to\clo_{\pm}\otimes(\bigwedge^6V^*)^{\otimes 2}$ is bijective.
\item $F(\clo_0)=\clo_3\otimes(\bigwedge^6V^*)^{\otimes 2}$.
\item $F=0$ on degenerate orbits $\clo_1$, $\clo_3$ and $\clo_6=\{0\}$.
\end{enumerate}
\end{prop}
It follows that the nontrivial part of $F$ is concentrated in the orbit $\clo_0$. Understanding the map $F:\clo_0\to\clo_3\otimes (\bigwedge^6V^*)^{\otimes 2}$ is equivalent to understanding the stabilizers of $\varphi\in\clo_0$ and $F(\varphi)\in\clo_3\otimes (\bigwedge^6V^*)^{\otimes 2}$ respectively.

For the simplicity of calculation, let us take $\varphi$ to be the normal form in our previous table. In addition, let us rename the vectors in $V^*$ by
\[\ud x^j=e^{2j-1},\quad \ud y^j=e^{2j},\quad\textrm{for }j=1,2,3.\]\
In this set-up, straightforward calculation yields
\[\begin{split}\varphi&=\ud x^1\wedge\ud y^2\wedge\ud y^3+\ud x^2\wedge\ud y^3\wedge\ud y^1+\ud x^3\wedge\ud y^1\wedge\ud y^2,\\
F(\varphi)&=4\ud y^1\wedge\ud y^2\wedge\ud y^3\otimes(\ud x^1\wedge\ud y^1\wedge\ud x^2\wedge\ud y^2\wedge\ud x^3\wedge\ud y^3).
\end{split}\]
Fix the basis $\{\ud x^1,\ud x^2,\ud x^3,\ud y^1,\ud y^2,\ud y^3\}$ of $V^*$ which we shall abbreviate as $\{\ud x,\ud y\}$. We have the following descriptions of the stabilizers of $F(\varphi)$ and $\varphi$.
\begin{prop}~
\begin{enumerate}
\item The stabilizer of $F(\varphi)$ consists of elements taking the matrix form
\[\begin{pmatrix}A & 0\\ B & C\end{pmatrix},\]
where $B$ is arbitrary, $A$ and $C$ are invertible $3\times3$ matrices satisfying $\det A\cdot\det^2C=1$.
\item The stabilizer of $\varphi$ is a subgroup of the stabilizer of $F(\varphi)$ whose elements further satisfy
\[A=\frac{C}{\det C},\textrm{ and }\tr(BC^{-1})=0.\]
\end{enumerate}
We also find that
\[K(\varphi)(\pt_{x^j})=0,\quad K(\varphi)(\pt_{y^j})=-2\pt_{x^j}\otimes(\ud x^1\wedge\ud y^1\wedge\ud x^2\wedge\ud y^2\wedge\ud x^3\wedge\ud y^3).\]
Consequently we have $\ker K(\varphi)=\img~K(\varphi)=\ker F(\varphi)=\spann\{\pt x\}.$
\end{prop}

People are also interested in the situation that $V$ is equipped with a symplectic form $\omega\in\bigwedge^2V^*$, in which case we shall consider the action of the group $\mathrm{Sp}(V,\omega)$ on the space of 3-forms. The Lefschetz decomposition
\[\bigwedge\nolimits^3V^*=\bigwedge\nolimits^3_0V^*\oplus\omega\wedge V^*\]
breaks the space of 3-forms on $V$ into two $\mathrm{Sp}(V,\omega)$-irreducible pieces, where $\bigwedge^3_0V^*$ denotes the space of $\omega$-primitive 3-forms which can be characterized by
\[\bigwedge\nolimits^3_0V^*=\{\varphi:\Lambda\varphi=0\}=\{\varphi:\omega\wedge\varphi=0\},\]
where $\Lambda$ denotes the Lefschetz operator of contraction with $\omega$.
 
Since any $\varphi\in\omega\wedge V^*$ is automatically a member of the orbit $\clo_1$, it is natural to only consider the $\mathrm{Sp}(V,\omega)$-orbits in $\bigwedge^3_0V^*$. In particular, for any $\mathrm{GL}(V)$-orbit $\clo$, we would like to know how $\clo\cap\bigwedge^3_0V^*$ breaks into distinct $\mathrm{Sp}(V,\omega)$-orbits. This problem has been studied by many authors such as \cite{lychagin1983, banos2003, bryant2006b}. We now know that $\clo_-\cap\bigwedge^3_0V^*$ splits into two families of $\mathrm{Sp}(V,\omega)$-orbits $\clo_-^\pm(\mu)$, each parametrized by a positive scaling factor $\mu$, and $\clo_+\cap\bigwedge^3_0V^*$ decomposes as a family of $\mathrm{Sp}(V,\omega)$-orbits $\clo_+(\mu)$ parametrized by a positive scalar $\mu$, and that $\clo_j\cap\bigwedge^3_0V^*$ breaks as the union of two $\mathrm{Sp}(V,\omega)$-orbits $\clo_j^\pm$ for $j=0,1$, and that $\clo_3\cap\bigwedge^3_0V^*$ and $\clo_6=\{0\}$ are single $\mathrm{Sp}(V,\omega)$-orbits. Taken from \cite{bryant2006b}, the normal forms of these orbits can be summarized as follows:
\[\begin{tabular}{|c|c|l|}
\hline
Orbit & Dimension & Normal Form \\
\hline
$\clo_-^\pm(\mu)$ & $13$ & $\mu(e^{135}-e^{146}\mp e^{236}\mp e^{245})$ \\
\hline
$\clo_+(\mu)$ & $13$ & $\mu(e^{135}+e^{246})$ \\
\hline
$\clo_0^\pm$ & $13$ & $e^{146}\pm e^{236}\pm e^{245}$ \\
\hline
$\clo_1^\pm$ & $10$ & $(e^{13}\mp e^{24})\wedge e^5$ \\
\hline
$\clo_3\cap\bigwedge^3_0V^*$ & $7$ & $e^{135}$ \\
\hline
$\clo_6=\{0\}$ & 0 & $0$\\
\hline
\end{tabular}\]
In all these cases,  we have $\omega=e^{12}+e^{34}+e^{56}$.

An immediate consequence is the following important observation:
\begin{prop}
For any $\varphi\in\clo_0^-\coprod\clo_0^+=\clo_0\cap\bigwedge^3_0V^*$, we have that $\ker K(\varphi)=\img~K(\varphi)=\ker F(\varphi)$ is a Lagrangian subspace of $(V,\omega)$ such that both the restrictions of $\varphi$ and $F(\varphi)$ on it vanish. 
\end{prop}

Under the presence of the symplectic form $\omega$, the 1-dimensional space $\bigwedge^6V^*$ is trivialized by $\dfrac{\omega^3}{3!}$, therefore for any primitive 3-form $\varphi$ we may canonically identify $K(\varphi)$, $F(\varphi)$ and $Q(\varphi)$ as elements in $\endo~V$, $\bigwedge_0^3V^*$ and $\rr$ respectively. Moreover, we have the following statement parallel to Proposition \ref{Fmap}.
\begin{prop}
The homogeneous cubic polynomial $F:\bigwedge_0^3V^*\to\bigwedge_0^3V^*$ satisfies
\begin{enumerate}
\item $F:\clo_-^\pm(\mu)\to\clo_-^\pm(4\mu^3)$ is bijective.
\item $F:\clo_+(\mu)\to\clo_+(2\mu^3)$ is bijective.
\item $F:\clo_0^\pm\to\clo_3\cap\bigwedge^3_0V^*$ is surjective.
\item $F\equiv0$ on $\clo_1^\pm$, $\clo_3\cap\bigwedge^3_0V^*$, and $\clo_6=\{0\}$.
\end{enumerate}
\end{prop}

By using $\omega$, there are a few ways to construct a symmetric bilinear form that is quadratic in $\varphi$. It is not surprising that they all yield the same quadratic form as detailed in the following proposition.
\begin{prop}
There is a natural symmetric bilinear form $q(\omega,\varphi)$ on $V$ associated to the pair $(\omega,\varphi)$, which satisfies
\[q(\omega,\varphi)(v_1,v_2)=\omega(v_1,K(\varphi)(v_2))=\frac{\iota_{v_1}\varphi\wedge\iota_{v_2}\varphi\wedge\omega}{\omega^3/3!}=-\omega(\iota_{v_1} \varphi,\iota_{v_2}\varphi).\]
\end{prop}
\begin{proof}
Since everything is algebraic and the union $\coprod_{\mu>0}\clo_-^+(\mu)$ is an open subset in $\bigwedge^3_0V^*$, we only need to prove this proposition for every member $\varphi\in\clo_-^+(\mu)$. This special case has been done in \cite[Lemma 5]{fei2021b}.
\end{proof}
The symmetric bilinear form $q(\omega,\varphi)$ has different signature depending on the orbital type of $\varphi$. To be more specific, we have:
\begin{prop}\label{cases}~
\begin{enumerate}
\item For $\varphi\in\clo_-^+(\mu)$, the symmetric bilinear form $q(\omega,\varphi)$ is positive definite, namely its signature is $(0,6,0)$.
\item For $\varphi\in\clo_-^-(\mu)$, the symmetric bilinear form $q(\omega,\varphi)$ is non-degenerate with signature $(0,2,4)$.
\item For $\varphi\in\clo_+(\mu)$, the symmetric bilinear form $q(\omega,\varphi)$ has split signature, namely its signature is $(0,3,3)$.
\item For $\varphi\in\clo_0^+$, the symmetric bilinear form $q(\omega,\varphi)$ has signature $(3,3,0)$. In particular the Lagrangian subspace $L:=\ker K(\varphi)=\img~K(\varphi)=\ker F(\varphi)$ is $q$-orthogonal to every vector in $V$. The induced bilinear form on $V/L$ is positive definite. As the map $K(\varphi):V/L\to L$ is an isomorphism, we also obtain a canonical positive definite metric on $L$.
\item For $\varphi\in\clo_0^-$, the symmetric bilinear form $q(\omega,\varphi)$ has signature $(3,1,2)$. In particular the Lagrangian subspace $L:=\ker K(\varphi)=\img~K(\varphi)=\ker F(\varphi)$ is $q$-orthogonal to every vector in $V$. The induced bilinear form on $V/L$ has signature $(0,1,2)$. As the map $K(\varphi):V/L\to L$ is an isomorphism, we also obtain a canonical metric of signature $(0,1,2)$ on $L$.
\item For $\varphi\in\clo_1^+$, the symmetric bilinear form $q(\omega,\varphi)$ has signature $(5,1,0)$. In particular the coisotropic subspace $\ker\varphi=\ker K(\varphi)$ is $q$-orthogonal to every vector in $V$. The induced bilinear form on the 1-dimensional space $V/\ker\varphi$ is positive definite.
\item For $\varphi\in\clo_1^-$, the symmetric bilinear form $q(\omega,\varphi)$ has signature $(5,0,1)$. In particular the coisotropic subspace $\ker\varphi=\ker K(\varphi)$ is $q$-orthogonal to every vector in $V$. The induced bilinear form on the 1-dimensional space $V/\ker\varphi$ is negative definite.
\item For $\varphi\in\clo_3\cap\bigwedge^3_0V^*$ or $\varphi=0$, the bilinear form $q(\omega,\varphi)$ vanishes. 
\end{enumerate}
\end{prop}

In the work \cite{fei2021b}, a primitive 3-form $\varphi$ is called positive if it is a member of the orbit $\clo_-^+(\mu)$ for some $\mu>0$. In this setting, the natural almost complex structure $J(\varphi)$ from Hitchin's construction is compatible with $\omega$ so the associated metric $g(\varphi)(\cdot,\cdot)=\omega(\cdot,J(\varphi)\cdot)$ is positive definite. We can define the norm of $\varphi$ through
\[|\varphi|^2\frac{\omega^3}{3!}=\varphi\wedge\hat\varphi,\]
which turns out to be the same as the norm of $\varphi$ measured under the metric $g(\varphi)$. In this set-up, the relations in (\ref{relations}) can be rewritten as \bea
\begin{split}K(\varphi)&=\frac{1}{2}|\varphi|^2J(\varphi),\\
F(\varphi)&=|\varphi|^2\hat\varphi,\\
Q(\varphi)&=-|\varphi|^4,\\
\mu&=\frac{1}{2}|\varphi|,\\
q(\omega,\varphi)&=\frac{|\varphi|^2}{2}g(\varphi).
\end{split}
\eea

\section{Integrability conditions}

From now on we shall always work under the presence of a symplectic form. Let $(M,\omega)$ be a connected symplectic 6-manifold, and let $\varphi$ be a smooth $\omega$-primitive 3-form on $M$. As explained in the previous section, we can naturally view $K(\varphi)$, $F(\varphi)$ and $Q(\varphi)$ as smooth sections of $\endo~TM$, $\bigwedge^3_0T^*M$ and $\bigwedge^0T^*M$ respectively, which further allows us to talk about various notions of integrability associated to $\varphi$. For the convenience of notation, we drop the reference to $\varphi$ in $K(\varphi)$, $F(\varphi)$ and $Q(\varphi)$ and similar quantities when there is no confusion.
\begin{dfn}~
\begin{enumerate}
\item The 3-form $\varphi$ is called \emph{integrable} or \emph{closed} if $\ud\varphi=0$. 
\item The 3-form $\varphi$ is called \emph{$K$-integrable} if the Nijenhuis tensor $N_K$ of $K(\varphi)$ vanishes. Here $N_K$ is defined by
\[N_K(X,Y)=-K^2[X,Y]+K([KX,Y]+[X,KY])-[KX,KY]\]
for any vector fields $X$ and $Y$ on $M$.
\item The 3-form $\varphi$ is called \emph{$F$-integrable} if $\ud F(\varphi)=0$.
\item The 3-form $\varphi$ is called \emph{$Q$-integrable} if $\ud Q(\varphi)=0$, namely $Q$ is a constant function on $M$.
\end{enumerate}
\end{dfn}
It is not surprising that these notions of integrability are related to each other. To explore this aspect, let us first establish a series of useful identities.

\begin{lemma}
For any vector fields $X$ and $Y$, we have
\bea
&&\iota_X\varphi\wedge F(\varphi)=-\varphi\wedge\iota_XF(\varphi)=\frac{1}{2}\iota_X(\varphi\wedge F(\varphi)),\label{o1}\\
&&\iota_X\varphi\wedge\iota_YF(\varphi)+\iota_Y\varphi\wedge\iota_XF(\varphi)=0,\label{o21}\\
&&\iota_Y\iota_X\varphi\wedge F(\varphi)=\varphi\wedge\iota_Y\iota_XF(\varphi)\label{o22}.
\eea
\end{lemma}
\begin{proof}
We shall prove this lemma by direct calculation. It is noteworthy that we have the following equations from the definition of $K$, $F$, and $Q$:
\bea
\iota_{KX}\frac{\omega^3}{3!}&=&-\iota_X\varphi\wedge\varphi,\label{K}\\
\iota_X F(\varphi)&=&-2\iota_{KX}\varphi,\label{F}\\
\varphi\wedge F(\varphi)&=&-Q(\varphi)\frac{\omega^3}{3!}\label{Q},\\
K^2&=&\frac{Q(\varphi)}{4}\id\label{square}.
\eea
Using (\ref{F}) and (\ref{K}), we get
\[-\varphi\wedge\iota_XF(\varphi)=2\varphi\wedge\iota_{KX}\varphi=-2\iota_{K^2X}\frac{\omega^3}{3!}=-\frac{Q}{2}\iota_X\frac{\omega^3}{3!}.\]
Therefore
\[\iota_X\varphi\wedge F(\varphi)-\varphi\wedge\iota_XF(\varphi)=\iota_X(\varphi\wedge F(\varphi))=-\iota_XQ\frac{\omega^3}{3!}=-2\varphi\wedge\iota_XF(\varphi)\]
and we prove (\ref{o1}).

Take $\iota_Y$ of left and right sides of (\ref{o1}), we get
\[\frac{1}{2}\iota_Y\iota_X(\varphi\wedge F(\varphi))=\iota_Y(\iota_X\varphi\wedge F(\varphi))=\iota_Y\iota_X\varphi\wedge F(\varphi)+\iota_X\varphi\wedge\iota_YF(\varphi).\]
As this expression is anti-symmetric in $X$ and $Y$, we obtain (\ref{o21}).

Take $\iota_Y$ of left and middle sides of (\ref{o1}), we get
\[\iota_Y\iota_X\varphi\wedge F(\varphi)+\iota_X\varphi\wedge\iota_YF(\varphi)=-\iota_Y\varphi\wedge\iota_XF(\varphi)+\varphi\wedge\iota_Y\iota_X F(\varphi).\]
Plug in (\ref{o21}) we prove (\ref{o22}).
\end{proof}
Next we establish an identity relating the Nijenhuis tensor $N_K$ with other quantities like $\ud\varphi$ and $\ud F(\varphi)$.
\begin{thm}
The Nijenhuis tensor $N_K$ satisfies
\bea\label{id}
\iota_{N_K(X,Y)}\frac{\omega^3}{3!}&=&\iota_Y\iota_X\ud\varphi\wedge F(\varphi)-\ud\varphi\wedge\iota_Y\iota_X F(\varphi)+2\varphi\wedge(\iota_Y\iota_{KX}-\iota_X\iota_{KY})\ud\varphi\nonumber\\
&&+\varphi\wedge\iota_Y\iota_X\ud F(\varphi).
\eea
\end{thm} 
\begin{proof}
From (\ref{K}) and (\ref{F}) we know
\bea
&&-\iota_{[KX,KY]}\frac{\omega^3}{3!}=-\iota_{\cll_{KX}(KY)}\frac{\omega^3}{3!}\nonumber\\
&=&-\cll_{KX}\iota_{KY}\frac{\omega^3}{3!}+\iota_{KY}\cll_{KX}\frac{\omega^3}{3!} \nonumber\\
&=&-\cll_{KX}\iota_{KY}\frac{\omega^3}{3!} +\cll_{KY}\iota_{KX}\frac{\omega^3}{3!}-d\iota_{KY}\iota_{KX}\frac{\omega^3}{3!}\nonumber\\
&=&\cll_{KX}(\iota_Y\varphi\wedge\varphi)-\cll_{KY}(\iota_X\varphi\wedge\varphi)-d\iota_{KY}\iota_{KX}\frac{\omega^3}{3!}\nonumber\\
&=&\left(\cll_{KX}\iota_Y\varphi-\cll_{KY}\iota_X\varphi\right)\wedge\varphi-d\iota_{KY}\iota_{KX}\frac{\omega^3}{3!} +\iota_Y\varphi\wedge\iota_{KX}\ud\varphi\nonumber\\
&&-\iota_X\varphi\wedge\iota_{KY}\ud\varphi+\frac{1}{2}\iota_X\varphi\wedge\ud\iota_YF(\varphi)-\frac{1}{2}\iota_Y\varphi\wedge\ud\iota_XF(\varphi).\label{t1}
\eea
Similarly, we have that
\bea
&&\iota_{K([KX,Y]+[X,KY])}\frac{\omega^3}{3!}=-\iota_{[KX,Y]}\varphi\wedge\varphi-(X\leftrightarrow Y)\nonumber\\
&=&(-\cll_{KX}\iota_Y\varphi+\iota_Y\cll_{KX}\varphi)\wedge\varphi-(X\leftrightarrow Y)\nonumber\\
&=&\left(-\cll_{KX}\iota_Y\varphi+\iota_Y\iota_{KX}\ud\varphi-\frac{1}{2}\iota_Y\ud\iota_XF(\varphi)\right)\wedge\varphi-(X\leftrightarrow Y),\label{t2}
\eea
and
\bea
&&-\iota_{K^2[X,Y]}\frac{\omega^3}{3!}=\iota_{K[X,Y]}\varphi\wedge\varphi=-\frac{1}{2}\iota_{[X,Y]}F(\varphi)\wedge\varphi\nonumber\\
&=&-\frac{1}{2}(\cll_X\iota_YF(\varphi)-\iota_Y\cll_XF(\varphi))\wedge\varphi\nonumber\\
&=&-\frac{1}{2}(\ud\iota_X\iota_Y+\iota_X\ud\iota_Y-\iota_Y\ud\iota_X)F(\varphi)\wedge\varphi+\frac{1}{2}\iota_Y\iota_X\ud F(\varphi)\wedge\varphi.\label{t3}
\eea
Adding (\ref{t1}), (\ref{t2}) and (\ref{t3}) together, we get:
\bea
&&\iota_{N_K(X,Y)}\frac{\omega^3}{3!}\nonumber\\
&=&-\frac{1}{2}\ud\iota_{KY}\iota_{KX}\frac{\omega^3}{3!}+\frac{1}{2}\iota_X\varphi\wedge\ud\iota_YF(\varphi) -\frac{1}{4}\ud\iota_X\iota_YF(\varphi)\wedge\varphi\label{int}\\
&&+\iota_Y\varphi\wedge\iota_{KX}\ud\varphi+\iota_Y\iota_{KX}\ud\varphi\wedge\varphi+\frac{1}{4}\iota_Y\iota_X\ud F(\varphi)\wedge\varphi-(X\leftrightarrow Y).\nonumber
\eea
Notice that
\bea
&&\iota_{KY}\iota_{KX}\frac{\omega^3}{3!}=-\iota_{KY}(\iota_X\varphi\wedge\varphi)\nonumber\\
&=&\iota_X\iota_{KY}\varphi\wedge\varphi-\iota_X\wedge\iota_{KY}\varphi\nonumber\\
&=&-\frac{1}{2}\iota_X\iota_YF(\varphi)\wedge\varphi+\frac{1}{2}\iota_X\varphi\wedge\iota_YF(\varphi),\nonumber
\eea
Consequently we have
\bea
-\frac{1}{2}\ud\iota_{KY}\iota_{KX}\frac{\omega^3}{3!}&=&\frac{1}{4}\ud\iota_X\iota_YF(\varphi)\wedge\varphi +\frac{1}{4}\iota_X\iota_YF(\varphi)\wedge\ud\varphi\nonumber\\
&&-\frac{1}{4}\ud\iota_X\varphi\wedge\iota_YF(\varphi)-\frac{1}{4}\iota_X\varphi\wedge\ud\iota_YF(\varphi).\label{kk}
\eea
Plug (\ref{kk}) in (\ref{int}), we obtain
\bea
&&\iota_{N_K(X,Y)}\frac{\omega^3}{3!}\nonumber\\
&=&\frac{1}{4}\iota_X\varphi\wedge\ud\iota_YF(\varphi)-\frac{1}{4}\ud\iota_X\varphi\wedge\iota_YF(\varphi)+\frac{1}{4}\iota_X\iota_YF(\varphi)\wedge\ud\varphi \nonumber\\
&&+\iota_Y\varphi\wedge\iota_{KX}\ud\varphi+\iota_Y\iota_{KX}\ud\varphi\wedge\varphi+\frac{1}{4}\iota_Y\iota_X\ud F(\varphi)\wedge\varphi-(X\leftrightarrow Y)\nonumber\\
&=&\frac{1}{4}\iota_X\varphi\wedge(\cll_YF(\varphi)-\iota_Y\ud F(\varphi))-\frac{1}{4}(\cll_X\varphi-\iota_X\ud\varphi)\wedge\iota_YF(\varphi)+\frac{1}{4}\iota_X\iota_YF(\varphi)\wedge\ud\varphi\nonumber\\
&&+\frac{1}{4}\iota_Y\iota_X\ud F(\varphi)\wedge\varphi+\iota_Y\varphi\wedge\iota_{KX}\ud\varphi+\iota_Y\iota_{KX}\ud\varphi\wedge\varphi-(X\leftrightarrow Y)\label{crucial}.
\eea
Observe that
\bea
&&\iota_X\varphi\wedge\cll_YF(\varphi)-\cll_X\varphi\wedge\iota_YF(\varphi)-(X\leftrightarrow Y)\nonumber\\
&=&\iota_X\varphi\wedge\cll_YF(\varphi)+\cll_Y\varphi\wedge\iota_XF(\varphi)-(X\leftrightarrow Y)\nonumber\\
&=&\cll_Y(\iota_X\varphi\wedge F(\varphi)+\varphi\wedge\iota_XF(\varphi))-\cll_Y\iota_X\varphi\wedge F(\varphi)-\varphi\wedge\cll_Y\iota_XF(\varphi)-(X\leftrightarrow Y)\nonumber\\
&=&(\cll_X\iota_Y-\cll_Y\iota_X)\varphi\wedge F(\varphi)+\varphi\wedge(\cll_X\iota_Y-\cll_Y\iota_X)F(\varphi),
\eea
where we used that $\iota_X\varphi\wedge F(\varphi)+\varphi\wedge\iota_XF(\varphi)=0$ from (\ref{o1}). Since
\[\cll_X\iota_Y-\cll_Y\iota_X=\iota_{[X,Y]}+\iota_Y\cll_X-\cll_Y\iota_X=\iota_{[X,Y]}+\iota_Y\iota_X\ud-\ud\iota_Y\iota_X,\]
we get
\bea
&&\iota_X\varphi\wedge\cll_YF(\varphi)-\cll_X\varphi\wedge\iota_YF(\varphi)-(X\leftrightarrow Y)\nonumber\\
&=&\iota_Y\iota_X\ud\varphi\wedge F(\varphi)+\varphi\wedge\iota_Y\iota_X\ud F(\varphi)-\ud\iota_Y\iota_X\varphi\wedge F(\varphi)-\varphi\wedge\ud\iota_Y\iota_X F(\varphi).\nonumber
\eea
By taking a $\ud$ of (\ref{o22}), we get
\[\ud\iota_Y\iota_X\varphi\wedge F(\varphi)+\varphi\wedge\ud\iota_Y\iota_X F(\varphi)=\iota_Y\iota_X\varphi\wedge\ud F(\varphi)+\ud\varphi\wedge\iota_Y\iota_XF(\varphi),\]
hence
\bea
&&\iota_X\varphi\wedge\cll_YF(\varphi)-\cll_X\varphi\wedge\iota_YF(\varphi)-(X\leftrightarrow Y)\nonumber\\
&=&\iota_Y\iota_X\ud\varphi\wedge F(\varphi)+\varphi\wedge\iota_Y\iota_X\ud F(\varphi)-\iota_Y\iota_X\varphi\wedge\ud F(\varphi)-\ud\varphi\wedge\iota_Y\iota_XF(\varphi).\nonumber
\eea
Substitute it back into (\ref{crucial}), we get
\bea
&&\iota_{N_K(X,Y)}\frac{\omega^3}{3!}\nonumber\\
&=&\frac{1}{8}(\iota_Y\iota_X\ud\varphi\wedge F(\varphi)+\varphi\wedge\iota_Y\iota_X\ud F(\varphi)-\iota_Y\iota_X\varphi\wedge\ud F(\varphi)-\ud\varphi\wedge\iota_Y\iota_XF(\varphi))\nonumber\\
&&-\frac{1}{4}\iota_X\varphi\wedge\iota_Y\ud F(\varphi)+\frac{1}{4}\iota_X\ud\varphi\wedge\iota_YF(\varphi)-\frac{1}{4}\ud\varphi\wedge\iota_Y\iota_XF(\varphi)+\frac{1}{4}\varphi\wedge\iota_Y\iota_X\ud F(\varphi)\nonumber\\
&&+\iota_Y\varphi\wedge\iota_{KX}\ud\varphi+\iota_Y\iota_{KX}\ud\varphi\wedge\varphi-(X\leftrightarrow Y).\label{tosimp}
\eea
By taking double interior product of the equations $\ud\varphi\wedge F(\varphi)=0$, $\varphi\wedge\ud F(\varphi)=0$ and $\varphi\wedge\ud\varphi=0$, we get
\bea
\iota_X\ud\varphi\wedge\iota_YF(\varphi)-(X\leftrightarrow Y)&=&\iota_Y\iota_X\ud\varphi\wedge F(\varphi)+\ud\varphi\wedge\iota_Y\iota_XF(\varphi),\nonumber\\
-\iota_X\varphi\wedge\iota_Y\ud F(\varphi)-(X\leftrightarrow Y)&=&\iota_Y\iota_X\varphi\wedge\ud F(\varphi)+\varphi\wedge\iota_Y\iota_X\ud F(\varphi),\nonumber\\
\iota_Y\varphi\wedge\iota_{KX}\ud\varphi-\iota_Y\iota_{KX}\ud\varphi\wedge\varphi &=&-\frac{1}{2}\iota_Y\ud\varphi\wedge\iota_XF(\varphi)-\frac{1}{2}\ud\varphi\wedge\iota_Y\iota_XF(\varphi).\nonumber
\eea
From these identities we can conclude that
\bea
&&\iota_{N_K(X,Y)}\frac{\omega^3}{3!}\nonumber\\
&=&\varphi\wedge\iota_Y\iota_X\ud F(\varphi)+\iota_Y\iota_X\ud\varphi\wedge F(\varphi)-\ud\varphi\wedge\iota_Y\iota_X F(\varphi)+2\varphi\wedge(\iota_Y\iota_{KX}-\iota_X\iota_{KY})\ud\varphi.\nonumber
\eea
\end{proof}
As a corollary, we have proved
\begin{cor}
If $\varphi$ is integrable and $F$-integrable, then it is $K$-integrable.
\end{cor}
\begin{rmk}\label{rmk}
Notice that $K$-integrability implies that
\[[KX,KY]=K([KX,Y]+[X,KY]-K[X,Y])\in\img~K.\]
It follows from the Frobenius theorem that the distribution $\img~K(\varphi)$ is integrable. However, it seems that $K$-integrability is a stronger condition than that $\img~K(\varphi)$ is integrable as a distribution. For example, if $\varphi$ belongs to the orbit $\clo_0^\pm$ pointwise, then $\img~K(\varphi)=\ker F(\varphi)$, hence the integrability of $\img~K(\varphi)$ is guaranteed solely by the $F$-integrability of $\varphi$. 
\end{rmk} 

Next, we also prove that
\begin{thm}
If $\varphi$ is integrable and $F$-integrable, then it is $Q$-integrable.
\end{thm}
\begin{proof}
We need to prove that $Q(\varphi)$ is a constant on $M$ under the condition that $\ud\varphi=\ud F(\varphi)=0$. We can decompose $M$ as
\[M=\{Q<0\}\coprod\{Q=0\}\coprod\{Q>0\}.\]
As $\{Q<0\}$ and $\{Q>0\}$ are both open subsets of $M$, we only need to show that in every connected component of the set $\{Q\neq 0\}$, $Q$ is a constant. Let us assume that $U$ is a connected open subset of $\{Q<0\}$. In this case, we know that $\varphi$ defines an almost complex structure $J$ compatible with $\omega$ and that $F(\varphi)=|\varphi|^2\hat\varphi$ on $U$. Since $\ud\varphi=0$, in the language of \cite{fei2021b}, we get Type IIA structure $(\omega,\varphi)$ on $M$. In \cite[pp. 791-792]{fei2021b} we have shown that $\ud\hat\varphi$ is a $(2,2)$-form with respect to $J$. Consequently, we know that
\[\ud F(\varphi)=\ud|\varphi|^2\wedge\hat\varphi+|\varphi|^2\ud\hat\varphi\]
is exactly the decomposition of $\ud F(\varphi)$ into its $(3,1)+(1,3)$ and $(2,2)$ components. Therefore $\ud F(\varphi)=0$ implies that both $\ud|\varphi|^2=0$ and $\ud\hat\varphi=0$, which further implies that $Q=-|\varphi|^4$ is a constant and $J$ is integrable.

The situation that $U$ is contained in $\{Q>0\}$ is parallel, see the analogous construction outlined in \cite{fei2015b}.
\end{proof}
To summarize, we have proved that
\begin{thm}\label{integ}
If $\varphi$ is integrable and $F$-integrable, then it is also $K$-integrable and $Q$-integrable. Moreover, $\varphi$ is either pointwise stable or pointwise unstable. It is impossible for $\varphi$ to be stable at some points while unstable at other places. 
\end{thm}
\begin{rmk}
There are examples of $\varphi$ being both integrable and $F$-integrable with $Q(\varphi)=0$ everywhere such that $\varphi$ belongs to different non-stable orbits at different points.
\end{rmk}

\begin{rmk}
Being both integrable and $F$-integrable is a very natural condition to impose. For example, when $\varphi$ belongs to the orbit $\clo_-^+(\mu)$ pointwise, being both integrable and $F$-integrable is equivalent to that the associated metric $q(\omega,\varphi)$ is both K\"ahler and Ricci-flat. If the underlying manifold is a K\"ahler Calabi-Yau 3-fold, such a $\varphi$ can always be found by running the Type IIA flow with the real part of any holomorphic volume form as initial data, see \cite[Theorem 9]{fei2021b}. This condition also fits well into what Hitchin calls the ``nonlinear Hodge theory''. In fact, on a closed symplectic 6-manifold $(M,\omega)$, one can introduce the a functional $I$ on the space of closed 3-forms within a fixed cohomology class:
\[I(\varphi)=\int_MQ(\varphi)\frac{\omega^3}{3!}.\]
One can show that the critical points of $I$ are exactly those 3-forms that are both integrable and $F$-integrable. Therefore it is naturally to introduce the following definition in the context of nonlinear Hodge theory.
\begin{dfn}
We say a primitive 3-form on $(M,\omega)$ is $F$-harmonic if it is both integrable and $F$-integrable.
\end{dfn}
\end{rmk}
 
\section{Degeneration of Calabi-Yau 3-folds and the SYZ conjecture}

In this section, we introduce a version of degeneration of Calabi-Yau structures and discuss its potential interactions with the Stroinger-Yau-Zaslow (SYZ) conjecture \cite{strominger1996}.\\

Roughly speaking, the SYZ conjecture predicts that near a large complex structure limit, which is a place where the complex structure of a Calabi-Yau manifold has maximal degeneration, the Calabi-Yau manifold admits a fibration of special Lagrangian tori with singularities. In addition, the base of the fibration carries a pair of dual Hessian structures (with singularities!) such that the local potentials solve a real Monge-Amp\`ere equation. Moreover, the mirror Calabi-Yau manifold could be constructed by dualizing the torus fibration while incorporating quantum corrections. The SYZ conjecture has served as a guiding principle to understand the mirror symmetry from a geometric point of view. There has been numerous work devoted to various aspects of the SYZ conjecture. We recommend the survey papers of Gross-Siebert \cite{gross2016} and Yang Li \cite{li2022b} and references therein for its recent developments.

The usual approach to the SYZ conjecture factors through the semi-flat geometry, which can be viewed as an asymptotic model of the Ricci-flat K\"ahler metrics in the regular part near a large complex structure limit. In this and the next section, we shall consider a weaker version of the SYZ picture, and answer the following question: How a Lagrangian foliation naturally arises from certain degenerations of complex structures on a 6-dimensional Calabi-Yau manifold, where semi-flatness is an automatic consequence of the special structures of ambient geometry.

The standard set-up for the SYZ conjecture is to consider a polarized degeneration of Calabi-Yau manifolds $\pi:(\mathcal{X},\mathcal{L})\to D$, where $D$ is the unit disc in $\ccc$ and $\mathcal{L}$ is a relative ample line bundle over $\mathcal{X}$. For $t\neq 0\in D$, $X_t:=\pi^{-1}(t)$ is a smooth connected Calabi-Yau manifold polarized by $\mathcal{L}|_{X_t}$. By choosing a suitable Hermitian metric on $\mathcal{L}$, the curvature of $\mathcal{L}$ defines a fiberwise K\"ahler metric $\omega_t$ on $X_t$. It is well-known that locally away from $t=0$, the map $\pi:\mathcal{X}\to D$ is a locally trivial fibration in the smooth category and under such an identification all $\omega_t$ have the same de Rham cohomology class.

In this paper, we shall adopt a slightly different version of degeneration of Calabi-Yau manifolds. For us, a degeneration of Calabi-Yau manifolds is a family of holomorphic volume forms $\{\Omega_t\}_{0<t\leq1}$ on a fixed symplectic manifold $(M,\omega)$. For any path $\alpha:[0,1]\to D$ with $\alpha(0)=0$, we may get a degeneration of Calabi-Yau manifolds in our sense from a polarized degeneration $\pi:(\mathcal{X},\mathcal{L})\to D$ as follows.

By an abuse of notations, let us identify $t\in[0,1]$ with $\alpha(t)\in D$ when necessary. Given $\alpha$ and $\pi$, we let $M=X_1$ be equipped with its associated K\"ahler form $\omega_1$. As the interval $(0,1]$ is contractible, there exists a smooth identification of $\pi|_{\pi^{-1}((0,1])}\to(0,1]$ with the natural projection $M\times(0,1]\to(0,1]$. Under such an identification, we get a family of symplectic forms $\{\omega_t\}_{0<t\leq 1}$ on $M$, all living in the same de Rham cohomology class. By the Moser's trick \cite{moser1965}, we can find a smooth family of diffeomorphisms $f_t$ of $M$ with $f_1=\textrm{Id}_M$ such that
\[f_t^*\omega_t=\omega_1\]
for all $0<t\leq 1$. In this way, we have shown that $(X_t,\omega_t)\cong(M,\omega:=\omega_1)$ under a choice suitable symplecticmorphisms. As the relative canonical bundle of $\pi$ is trivial, we may take a section of it as a family of holomorphic volume forms on $X_t$, which in turn gives rise to a family of holomorphic volume forms $\{\Omega_t\}_{0<t\leq 1}$ on $(M,\omega)$ under the above identification.

\begin{dfn}
A degeneration of Calabi-Yau structures $\{\Omega_t\}_{0<t\leq 1}$ on $(M,\omega)$ is called \emph{Ricci-flat}, if for every $t\in(0,1]$, the K\"ahler metric determined by $(M,\omega,\Omega_t)$ is Ricci-flat.
\end{dfn} 

Essentially, every degeneration of Calabi-Yau structures on $(M,\omega)$ is Ricci-flat due to the following result.

\begin{thm}
Let $\{\Omega_t\}_{0<t\leq 1}$ be a degeneration of Calabi-Yau structures on a symplectic 6-manifold $(M,\omega)$. Then there is a family of diffeomorphisms $f_t$ for $0<t\leq 1$ such that $\{f_t^*\Omega_t\}_{0<t\leq 1}$ is a Ricci-flat degeneration of Calabi-Yau structures on $(M,\omega)$.
\end{thm}
\begin{proof}
Write $\Omega_t=\varphi_t+i\hat\varphi_t$, where $\varphi_t$ and $\hat\varphi_t$ are the real and imaginary parts of $\Omega_t$ respectively. From Section 2, we know that all the information encoded in the degeneration of complex structures can be recovered from the family $\varphi_t$, which is a family of real primitive 3-forms on $M$ belonging to the orbit $\clo_-^+(\mu)$ pointwise. As the $\Omega_t$ is holomorphic with respect to the associated integrable complex structure, we have $\ud\varphi_t=\ud\hat\varphi_t=0$. By running the Type IIA flow \cite[Theorem 9]{fei2021b} using $\varphi_t$ as the initial data on $(M,\omega)$, we know there exists a diffeomorphism $f_t$ of $(M,\omega)$, isotopic to identity, such that the K\"ahler metric (or the Type IIA structure, in the language of \cite{fei2021b}) associated to the triple $(M,\omega,f_t^*\varphi_t)$, is Ricci-flat.
\end{proof}
\begin{rmk}
From the viewpoint of the above proof and previous discussion, a Ricci-flat degeneration of Calabi-Yau structures $\{\Omega_t\}$ on $(M,\omega)$ is nothing but a family of $F$-harmonic 3-forms $\varphi_t$ belonging to the orbit $\clo_-^+(\mu)$ pointwise.
\end{rmk}

It is then natural to introduce the following definition:
\begin{dfn}
A wonderful degeneration of Calabi-Yau structures $\{\Omega_t\}_{0<t\leq 1}$ on a symplectic 6-manifold $(M,\omega)$ is a Ricci-flat degeneration of Calabi-Yau structures such that the convergence
\[\lim_{t\to 0}\varphi_t=\varphi_0\]
holds smoothly on an open dense subset $M^\circ$ of $M$ with the additional property that $\varphi_0$ belongs to the orbit $\clo_0^+$ pointwise on $(M^\circ,\omega)$. Here, $\varphi_t$ is the real part of the holomorphic volume form $\Omega_t$.
\end{dfn}

For a wonderful degeneration of Calabi-Yau manifolds, there are a few points worth mentioning.
\begin{rmk}~
\begin{itemize}
\item As the convergence $\lim_{t\to 0}\varphi_t=\varphi_0$ is smooth, the 3-form $\varphi_0$ is automatically $F$-harmonic on $(M^\circ,\omega)$.
\item In a way, it is helpful to think $M^\circ$ as the ``regular part'' of the degeneration of complex structures in consideration, whose geometry is determined by the triple $(M^\circ,\omega,\varphi_0)$. On the other hand, the ``singular locus'' $M\setminus M^\circ$ can be written as a disjoint union:
    \[M\setminus M^\circ=S_r\coprod S_s,\]
    where $S_s$ is the locus where $\lim_{t\to 0}\varphi_t$ does not exist, while $S_r$ is the subset where $\lim_{t\to 0}\varphi_t$ does exist but not belonging to the orbit $\clo_0^+$.
\item The condition that $\varphi_0$ belongs to the orbit $\clo_0^+$ is a very natural assumption. Because $\varphi_0$ is $F$-harmonic, we must have $Q(\varphi_0)$ is a constant on $M^\circ$. As $Q(\varphi_t)<0$ for all $t$, it is either $Q(\varphi_0)=0$ or $Q(\varphi_0)<0$. If $Q(\varphi_0)<0$, $\varphi_0$ would define a genuine Calabi-Yau structure which is not the kind of degeneration we expect. Therefore, the most interesting case is that $Q(\varphi_0)=0$ and in this scenario $\varphi_0$ must belong to the union of orbits $\clo_0^+\coprod\clo_1^+\coprod\left(\clo_3\cap\bigwedge^3_0V^*\right)\coprod\{0\}$. Our assumption is really about that $\varphi_0$ generically belongs to the most ``non-degenerate'' orbit among all possible degenerate orbits.
\item It is very important to consider the limit $\lim_{t\to0}\varphi_t$ instead of $\lim_{t\to0}\Omega_t$. As we shall see in the example below, it is possible for $\lim_{t\to0}\varphi_t$ to exist while $\lim_{t\to0}\Omega_t$ blowing up.
\end{itemize}
\end{rmk}

In order to understand a wonderful degeneration of Calabi-Yau structures, we shall first focus on the geometry associated to the triple $(M^\circ,\omega,\varphi_0)$, where $(M^\circ,\omega)$ is a symplectic 6-manifold, and $\varphi_0$ is an $F$-harmonic primitive 3-form on $M$ belonging to the orbit $\clo_0^+$ pointwise.

A direct consequence from Proposition \ref{cases} and Theorem \ref{integ} is that $\varphi_0$ naturally induces a Riemannian Lagrangian foliation on $(M^\circ,\omega)$ as stated in the following proposition.
\begin{prop}\label{Lag}
Let $(M^\circ,\omega)$ be a symplectic 6-manifold and $\varphi_0$ an $F$-harmonic primitive 3-form belonging to the orbit $\clo_0^+$ pointwise. Then $K=K(\varphi_0)$ is a smooth section of $\endo~TM^\circ$ such that $K^2=0$ and $\cll:=\ker K=\img~K$ defines a Lagrange foliation on $(M^\circ,\omega)$. In addition, the Nijenhuis tensor $N_K$ of $K$ vanishes, namely
\bea
N_K(X,Y)=K[KX,Y]+K[X,KY]-[KX,KY]=0\label{vanishing}
\eea
for any vector fields $X$ and $Y$. The natural symmetric bilinear form $q(\cdot,\cdot)=\omega(\cdot,K\cdot)$ induces a Riemannian metric on the vector bundle $TM^\circ/\cll$. The bundle $TM^\circ/\cll$ is canonically oriented so the volume form of the above Riemannian metric is exactly given by $\dfrac{1}{\sqrt{2}}F(\varphi_0)$. Through the bundle isomorphism $K:TM^\circ/\cll\cong \cll$, we also get a canonical Riemannian metric $g_\cll$ together with an orientation on the Lagrangian foliation $\cll$ such that
\[g_\cll(X,Y)=\omega(K^{-1}X,Y)\]
for any vector fields tangent to $\cll$. Clearly the definition of $g_\cll$ above is independent of the choice of preimage $K^{-1}X$ of $X$ since $\cll$ is Lagrangian.
\end{prop}

As a toy case, we present an example of wonderful degeneration on the 6-torus, where all the geometries are flat. A more sophisticated example will be analyzed in Section 6. 

\begin{ex}
Let $T=\rr^2/\zz^2$ be the standard 2-torus equipped with the symplectic form $\omega=\ud x\wedge\ud y$. For any $\tau$ in the upper half plane, consider the complex structure $J_\tau$ on $T$ defined by the holomorphic $(1,0)$-form $\ud x+\tau\ud y$ on $T$. It is not hard to see that as a Riemann surface, $(T,J_\tau)$ is biholomorphic to $\ccc/(\zz+\zz\tau)$, the elliptic curve with period $\tau$. For this reason we shall write $E_\tau$ to denote the Calabi-Yau 1-fold $(T,\omega,J_\tau)$. 

Now let $M_\tau=E_\tau\times E_\tau\times E_\tau$ with product complex and symplectic structures. It is easy to see that $M_\tau$ is a Calabi-Yau 3-fold with K\"aher Ricci-flat metric. Moreover, it is well-known that the family $M_\tau$ with $\tau\to\sqrt{-1}\infty$ is a large complex structure limit. For example, it is indicated in \cite{leung2005} that under the limit $\tau\to\sqrt{-1}\infty$, the geometry of $M_\tau$ collapses to a 3-dimensional base.

To fit in our framework, we also work with the parameter $t=\dfrac{\sqrt{-1}}{\tau}$, where $t>0$ is a real parameter. Consider the following holomorphic volume form $\Omega_t$ on $M_\tau$:
\[\Omega_t=\frac{1}{\tau^2}(\ud x^1+\tau\ud y^1)\wedge(\ud x^2+\tau\ud y^2)\wedge(\ud x^3+\tau\ud y^3),\]
with $\varphi_t=\textrm{Re}~\Omega_t$. In this set-up we have
\[\lim_{t\to0}\varphi_t=\ud x^1\wedge\ud y^2\wedge\ud y^3+\ud x^2\wedge\ud y^3\wedge\ud y^1+\ud x^3\wedge\ud y^1\wedge\ud y^2=:\varphi_0\]
belongs to the orbit $\clo_0^+$ at every point. Associated to $\varphi_0$ there is a natural Lagrangian foliation 
\[\mathcal{L}=\ker K(\varphi_0)=\spann\left\{\frac{\p}{\p x^1},\frac{\p}{\p x^2},\frac{\p}{\p x^3}\right\}.\]
In fact, in this example $\mathcal{L}$ is a Lagrangian torus fibration over a 3-torus base.
\end{ex}

As one can observe in this example, in order to have a wonderful degeneration, it is crucial to choose the right normalization for $\Omega_t$ hence $\varphi_t$. An alternative normalization would be $\varphi'_t=F(\varphi_t)=|\varphi_t|^2\hat\varphi_t$. For such a choice, the limit
\[\lim_{t\to 0}\varphi'_t=F(\varphi_0)\]
still exists, however it belongs to the orbit $\clo_3$. We shall see in the next section that the orbit $\clo_0^+$ leads to much richer geometric structures than the orbit $\clo_3$ does.

Another noteworthy phenomenon is that the limit $\lim_{t\to 0}\Omega_t$ does not exist because
\[\lim_{t\to 0}\hat\varphi_t=\lim_{t\to 0}\frac{F(\varphi_t)}{\sqrt{-Q(\varphi_t)}}\]
blows up as $F(\varphi_0)$ is nonzero and $\lim_{t\to 0}Q(\varphi_t)=Q(\varphi_0)=0$.

\section{Geometry of $(M^\circ,\omega,\varphi_0)$}

In this section, we study the geometry associated to the triple $(M^\circ,\omega,\varphi_0)$ where $(M^\circ,\omega)$ is a symplectic 6-manifold and $\varphi_0$ an $F$-harmonic 3-form belonging to the orbit $\clo_0^+$ pointwise. Such a geometry can be viewed as the Calabi-Yau version of the more general nilpotent Jordan(-K\"ahler) geometry developed in \cite{fei2024b}.

As seen in Proposition \ref{Lag}, a main player in this set-up is the Lagrangian foliation $\cll$ associated to $(M^\circ,\omega,\varphi_0)$. It turns out that there are very rich geometric structures on the leaves of the Lagrangian foliation $\cll$. In particular we have the following theorem.
\begin{thm}\label{structure}
Let $L$ be a leaf, namely a maximal integral (immersed) submanifold, of the Lagrangian foliation $\cll$. Then $L$ is a canonically oriented Riemannian manifold. Moreover, $L$ carries a canonical affine structure such that the natural Riemannian metric induced from $g_\cll$ is a Hessian metric with parallel volume form. In other words, in any local affine coordinate system on $L$, the metric is the Hessian of a locally defined convex function $h$ such that it solves a real Monge-Amp\`ere equation with constant right hand side.
\end{thm}
\begin{proof}
The first part follows from Proposition \ref{Lag} directly. Now let us introduce a partial connection $D$ on $\cll$ as follows. 

For any vector fields $X$ and $Y$ that are tangent to $\cll$ everywhere, we define
\bea
D_XY=K[X,K^{-1}Y].\label{dualbott}
\eea
Suppose we have two choices $Z_1$ and $Z_2$ of $K^{-1}Y$, since $K(Z_1)=K(Z_2)=Y$, we know that $Z_1-Z_2$ is a section of $\ker K=\cll$. As $\cll$ is integrable and $X$ is tangent to $\cll$, it follows that $[X,Z_1-Z_2]$ is also tangent to $\cll$ hence $K[X,Z_1-Z_2]=0$. Therefore $D$ is well defined as $D_XY$ does not depend on the choice of $K^{-1}Y$. One can check that $D$ satisfies the usual properties of a connection. For any function $f$, we have
\[D_{fX}Y=K[fX,K^{-1}Y]=K(f[X,K^{-1}Y])-(K^{-1}Yf)KX=fD_XY,\]
and
\[\begin{split}
D_X(fY)&=K[X,K^{-1}(fY)]=K[X,fK^{-1}Y]=K(f[X,K^{-1}Y])+Xf\cdot K(K^{-1}Y)\\
&=fD_XY+Xf\cdot Y.\end{split}\]
Moreover, by vanishing of the Nijenhuis tensor (\ref{vanishing}), we have
\[\begin{split}D_XY-D_YX=&K([K(K^{-1}X),K^{-1}Y]+[K^{-1}X,K(K^{-1}Y)])=[K(K^{-1}X),K(K^{-1}Y)]\\
=&[X,Y].\end{split}\]
The Jacobi identity further suggests that
\[\begin{split}&D_X(D_YZ)-D_Y(D_XZ)-D_{[X,Y]}Z\\
=&K[X,[Y,K^{-1}Z]]-K[Y,[X,K^{-1}Z]]-K[[X,Y],K^{-1}Z]\\
=&0.\end{split}\]
Consequently the restriction of $D$ to each leaf $L$ is a torsion-free and flat connection on $TL$, hence we get a canonical affine structure on every leaf.

Let $g_L$ be the induced Riemannian metric from $g_\cll$. Then for any tangent vectors $Y$ and $Z$ of $L$, we have $g_L(Y,Z)=\omega(K^{-1}Y,Z)$,
hence
\[D_Xg_L(Y,Z)=X\omega(K^{-1}Y,Z)-\omega([X,K^{-1}Y],Z)-\omega(K^{-1}Y,D_XZ).\]
On the other hand, $\omega$ is a symplectic form, so
\[\begin{split}0=\ud\omega(X,K^{-1}Y,Z)=&X\omega(K^{-1}Y,Z)+K^{-1}Y\omega(Z,X)+Z\omega(X,K^{-1}Y)\\
&-\omega([X,K^{-1}Y],Z)-\omega([K^{-1}Y,Z],X)-\omega([Z,X],K^{-1}Y),\end{split}\]
which can be translated to
\[D_Xg_L(Y,Z)=D_Zg_L(Y,X),\]
hence $Dg_L$ is a totally symmetric tensor. In a local affine coordinate system $\{x^1,x^2,x^3\}$, we can write $g_L=h_{ij}\ud x^i\otimes\ud x^j$. It follows that 
\[\frac{\p}{\p x^k}h_{ij}=\frac{\p}{\p x^j}h_{ik}\]
for any indices $i,j$ and $k$. Therefore locally we can find a potential function $h$ such that $h_{ij}=\dfrac{\p^2h}{\p x^i\p x^j}$ hence $g_L$ is a Hessian metric. Let $\vol_{g_L}$ be the volume form of the metric $g_L$. By a pointwise computation using the normal form, we have
\[\vol_{g_L}(X,Y,Z)=\frac{1}{\sqrt{2}}F(\varphi)(K^{-1}X,K^{-1}Y,K^{-1}Z).\]
Since $\ud F(\varphi)=0$ and $\cll=\ker F(\varphi)$, we get
\[\begin{split}&WF(\varphi)(K^{-1}X,K^{-1}Y,K^{-1}Z)-F(\varphi)([W,K^{-1}X],K^{-1}Y,K^{-1}Z)\\
&-F(\varphi)(K^{-1}X,[W,K^{-1}Y],K^{-1}Z)-F(\varphi)(K^{-1}X,K^{-1}Y,[W,K^{-1}Z])=0\end{split}\]
for any $X,Y,Z,W$ tangent to $\cll$. This equation translates to
\[D_W\vol_{g_L}(X,Y,Z)=0.\]
Therefore we know that $\vol_{g_L}$ is parallel under $D$.
\end{proof}

\begin{rmk}\label{foliationcoor}
It is not hard to write down a locally $D$-parallel frame of $TL$. By a theorem of Weinstein \cite{weinstein1971}, near any point $p$ of $M^\circ$, there exists locally a Darboux coordinate system $\{x^1,y^1,x^2,y^2,x^3,y^3\}$ such that
\[\omega=\ud x^1\wedge\ud y^1+\ud x^2\wedge\ud y^2+\ud x^3\wedge\ud y^3,\]
and the leaves of $\cll$ are locally defined by $y=(y^1,y^2,y^3)$ are constants. In this local coordinates, $\left\{K\dfrac{\p }{\p y^j}\right\}_{j=1}^3$ gives rise to a $D$-parallel frame on $TL$.
\end{rmk}

On each leaf $L$, the pair $(g_L,D)$ forms what is known as a Hessian structure \cite[Chap. 2]{shima2007}. By performing the Legendre transform locally, one can generate a dual Hessian structure $(g_L,D')$ on $L$, where $D'$ is a torsion-free and flat connection on $TL$ such that
\[D+D'=2\nabla,\]
where $\nabla$ is the Levi-Civita connection associated to $g_L$. The same metric $g_L$ is also Hessian with respect to $D'$ and the volume form $\vol_{g_L}$ is also parallel under $D'$. On the other hand, for any Lagrangian foliation, there is another canonical affine structure on its leaves characterized by the Bott (partial) connection $\nabla^B$ \cite{bott1972}. In this setting, see for example \cite[Chap. 4]{hamilton2023}, the Bott (partial) connection satisfies
\bea
\omega(\nabla^B_XY,Z)=X\omega(Y,Z)+\omega([X,Z],Y)\label{bott}
\eea
for any $X,Y$ tangent to $\cll$ and arbitrary vector field $Z$. It is not surprising to expect the following theorem:
\begin{thm}
The Bott connection $\nabla^B$ is dual to $D$ in the sense of Hessian structures, namely $\nabla^B=D'$. In particular, with respect to the affine structure determined by $\nabla^B$, the metric $g_L$ on each leaf $L$ is Hessian and its local potential functions all solve a real Monge-Amp\`ere equation with constant right hand side.
\end{thm}
\begin{proof}
We only need to prove that
\[Xg_L(Y,Z)=g_L(D_XY,Z)+g_L(Y,\nabla^B_XZ)\]
for any vector fields $X,Y,Z$ tangent to $\cll$. This follows directly from that $g_L(\cdot,\cdot)=\omega(K^{-1}\cdot,\cdot)$ and the definitions of $D$ (\ref{dualbott}) and $\nabla^B$ (\ref{bott}).
\end{proof}
\begin{thm}\label{sf}
The natural metric $g_L$ on any leaf $L$ has nonnegative Ricci curvature. If the leaf $L$ is compact, then it must be a flat torus.
\end{thm}
\begin{proof}
In an affine coordinate, the Ricci curvature $R_{jk}$ of a Hessian metric $g$ is given by
\bea
R_{jk}=\frac{1}{4}h^{st}h^{lp}(h_{jps}h_{klt}-h_{jkt}h_{pls}),\label{riccigen}
\eea
where $h$ is the local potential function for the Hessian metric $g$ and $h_{ijk}$ is the short-handed notation for $\dfrac{\p^3h}{\p x^i\p x^j\p x^k}$, the 3rd derivative of $h$ with respect to an affine coordinate system. If in addition that the volume form is parallel under the affine connection, then $\det[h_{pl}]$ is a constant, hence
\[\frac{\p}{\p x^s}\log\det[h_{pl}]=h^{lp}h_{pls}=0.\]
Plugging it in (\ref{riccigen}), we get
\bea
R_{jk}=\frac{1}{4}h^{st}h^{lp}h_{jps}h_{klt}
\eea
is semi-positive definite. In other words, for any vector field $X$, we have
\bea
\textrm{Ric}(X,X)=\frac{1}{4}|D_Xg|^2\geq0.\label{ricci}
\eea
It follows that the scalar curvature $S$ has the expression
\bea
S=\frac{1}{4}h^{st}h^{ik}h^{jl}h_{sil}h_{tkj}=\frac{1}{4}|Dg|^2\geq 0.\label{scalar}
\eea
As in Theorem (\ref{structure}) we have proved that for any leaf $L$, it has a Hessian structure with parallel volume form, so we know its Ricci curvature is nonegative from (\ref{ricci}).

Now suppose that $L$ is a compact leaf. It is well-known that $L$ must be a torus. By a famous theorem of Schoen-Yau \cite{schoen1979} and Gromov-Lawson \cite{gromov1980}, we know that any Riemannian metric on a torus with nonnegative scalar curvature must be flat.
\end{proof}
\begin{rmk}
In general, the two affine structures determined by $D$ and $\nabla^B$ are different. This follows from the fact that the metric $g_L$ on each leaf is in general not flat. See the example below. However, when $g_L$ is a flat metric, the two connections $D$ and $\nabla^B$ both coincide with the Levi-Civita connection of $g_L$.
\end{rmk}
\begin{ex}
Let $(N,g)$ be an oriented 3-dimensional Riemannian manifold. Over $N$, we consider its cotangent bundle $T^*N$, namely the bundle of 1-forms, and also the bundle of 2-forms $\bigwedge^2T^*N$. It is well-known that $T^*N$ carries a tautological 1-form, whose exterior derivative gives rise to the canonical symplectic structure $\omega$ on $T^*N$. Let $\{x^1,x^2,x^3\}$ be a local coordinate chart on $N$, and let $\{y_1,y_2,y_3\}$ be the corresponding coordinate in the cotangent direction, we know that $\omega$ takes the Darboux form
\[\omega=\ud x^1\wedge\ud y_1+\ud x^2\wedge\ud y_2+\ud x^3\wedge\ud y_3.\]
Similarly, the total space of $\bigwedge^2T^*N$ carries a tautological 2-form $\alpha$ defined as follows: Let $(p,\lambda)$ be a point in $\bigwedge^2T^*N$ where $p\in N$ and $\lambda\in\bigwedge^2T^*_pN$. For any $X,Y\in T_{(p,\lambda)}\bigwedge^2T^*N$, we define that
\[\alpha_{(p,\lambda)}(X,Y)=\lambda(\pi_*X,\pi_*Y),\]
where $\pi:\bigwedge^2T^*N\to N$ is the projection map. Over a local coordinate chart $\{x^1,x^2,x^3\}$ of $N$, any 2-form on $N$ can be written as $t^1\ud x^2\wedge\ud x^3+t^2\ud x^3\wedge\ud x^1+t^3\ud x^1\wedge\ud x^2$. If we use $\{x^1,t^1,x^2,t^2,x^3,t^3\}$ as a local coordinate chart on $\bigwedge^2T^*N$, we can express $\alpha$ as
\[\alpha=t^1\ud x^2\wedge\ud x^3+t^2\ud x^3\wedge\ud x^1+t^3\ud x^1\wedge\ud x^2.\]
By taking the exterior derivative, we get the canonical 3-form $\ud\alpha$ on $\bigwedge^2T^*N$:
\[\begin{split}\ud\alpha&=\ud(t^1\ud x^2\wedge\ud x^3+t^2\ud x^3\wedge\ud x^1+t^3\ud x^1\wedge\ud x^2)\\
&=\ud t^1\wedge\ud x^2\wedge\ud x^3+\ud t^2\wedge\ud x^3\wedge\ud x^1+\ud t^3\wedge\ud x^1\wedge\ud x^2.\end{split}\]

The Riemannian metric $g$ on $N$ provides us an isomorphism $T^*N\cong\bigwedge^2T^*N$ via the Hodge star operator. In the local coordinates above, we have
\[t^j=y_kg^{kj}\sqrt{\det g},\]
or
\[y_k=\frac{1}{\sqrt{\det g}}g_{kj}t^j.\]
Under this identification, we can think of $\omega$ as a symplectic form on $\bigwedge^2T^*N$ with the expression
\[\omega=\frac{g_{kj}}{\sqrt{\det g}}\ud x^k\wedge\ud t^j+t^j\ud x^k\wedge\ud\left(\frac{g_{kj}}{\sqrt{\det g}}\right).\]
It is not hard to find out that
\[\frac{\omega^3}{3!}=\frac{1}{\sqrt{\det g}}\ud x^1\wedge\ud t^1\wedge\ud x^2\wedge\ud t^2\wedge\ud x^3\wedge\ud t^3.\]
Let $r$ be the norm square function on $\bigwedge^2T^*N$, namely $r(p,\lambda)=|\lambda|^2_{g_p}$. In local coordinates we have that
\[r(x,t)=\frac{t^jt^k}{\det g}g_{jk}.\]
A straightforward calculation shows that 
\bea
\alpha\wedge\omega=\frac{\sqrt{\det g}}{2}\ud x^1\wedge\ud x^2\wedge\ud x^3\wedge\ud r.\label{basic}
\eea
It follows that $\ud\alpha\wedge\omega=\ud(\alpha\wedge\omega)=0$, hence $\ud\alpha$ is primitive.

Let $f$ be a smooth single variable function of $r$, we can define
\[\varphi_f=\ud(f\alpha)=f\ud\alpha+f'\ud r\wedge\alpha.\]
It follows from (\ref{basic}) that $\varphi_f$ is a primitive 3-form for any such $f$. Moreover, direct calculation reveals that
\[\begin{split}&K(\varphi_f)\left(\frac{\p}{\p t^j}\right)=0,\\
&K(\varphi_f)\left(\frac{\p}{\p x^j}\right)=2f\sqrt{\det g}\left[(f+2rf')\frac{\p}{\p t^j}-f'\frac{\p r}{\p t^j}\cle\right],\end{split}\]
where $\cle=t^j\dfrac{\p}{\p t^j}$ is the local Euler vector field. It follows that 
\[F(\varphi_f)=-4f^2(f+2rf')\sqrt{\det g}\cdot\ud x^1\wedge\ud x^2\wedge\ud x^3.\]
We are mainly interested in the case that 
\bea
f(r)=r^{-\frac{1}{2}}(r^{\frac{3}{2}}+C)^{\frac{1}{3}}\label{fchoice}
\eea
for some constant $C$, so that
\[f^2(f+2rf')=1.\]
In this case, we have
\[F(\varphi_f)=-4\sqrt{\det g}\cdot\ud x^1\wedge\ud x^2\wedge\ud x^3\]
is a closed 3-form, hence $\varphi_f$ is $F$-harmonic.

With the above choice (\ref{fchoice}) of $f$, we consider $M=\bigwedge^2T^*N$ when $C=0$ or $M=\bigwedge^2T^*N\setminus N$ when $C>0$ and $M=\bigwedge^2T^*N
\setminus\{r\leq(-C)^{2/3}\}$ when $C<0$. It follows that $(M,\omega,\varphi_f)$ is a triple such that $(M,\omega)$ is a symplectic 6-manifold and $\varphi_f$ is a primitive 3-form on $M$ belonging to the orbit $\clo_0^+$ pointwise. In addition, $\varphi_f$ is $F$-harmonic.

It is obvious that the Lagrangian foliation $\cll$ in this case is simply given by the vertical foliation of the fibration $\pi:\bigwedge^2T^*N\to N$. On one hand, the affine structure associated to the Bott connection $\nabla^B$ is the standard affine structure on the fibers of the vector bundle $\pi:\bigwedge^2T^*N\to N$. To understand the affine structure associated to $D$, we let 
\[V_j=K(\varphi_f)\left(\frac{\p}{\p x^j}\right)=2f\sqrt{\det g}\left[(f+2rf')\frac{\p}{\p t^j}-f'\frac{\p r}{\p t^j}\cle\right],\]
for $j=1,2,3$. By Remark \ref{foliationcoor} we know that the restriction of $V_j$ to any leaf is parallel under $D$, hence gives rise to a local affine coordinate on the Lagrangian leaf. In this local affine coordinate, the induced metric $g_{\cll}$ takes the Hessian form
\[h_{jk}=2f^{-1}g_{jk}-ff'\det g\frac{\p r}{\p t^j}\frac{\p r}{\p t^k}.\]
By direct computation we have
\[\det [h_{jk}]=8\det g\]
is constant in $t$, as predicted. In addition, we have
\[h^{jk}=\frac{f}{2}\left[g^{jk}+\frac{f'\det g}{2f}g^{jp}\frac{\p r}{\p t^p}\frac{\p r}{\p t^q}g^{qk}\right],\]
and
\[h_{jkl}=-4f'\sqrt{\det g}\left(g_{jk}\frac{\p r}{\p t^l}+g_{kl}\frac{\p r}{\p t^j}+g_{lj}\frac{\p r}{\p t^k}\right)-\frac{C(5r^{\frac{3}{2}}+2C)}{2r^4(r^{\frac{3}{2}}+C)^{\frac{2}{3}}}(\det g)^{\frac{3}{2}}\frac{\p r}{\p t^j}\frac{\p r}{\p t^k}\frac{\p r}{\p t^l}.\]
To demonstrate that $D$ and $\nabla^B$ induces different affine structures on the leaves, we only need to show that the induced metric $g_L$ on the leaves are not flat, as $D$ and $\nabla^B$ define dual Hessian structures on every leaf. This can be done by computing the scalar curvature of every leaf $L$. By (\ref{scalar}) we know that
\[S=\frac{1}{4}|Dg|^2>0\]
because $h_{jkl}\neq 0$.

Alternatively, we notice that in this example the induced metric on the leaves are rotationally symmetric. If we let $r=\rho^2$, then this metric can be written as
\[g_L=\frac{\rho^2}{2(\rho^3+C)^{2/3}}(\ud\rho)^2+\frac{\rho(\rho^3+C)^{1/3}}{2}g_{S^2}.\]
It follows that the scalar curvature of $g_L$ is
\[S=\frac{5C^2}{\rho^4(\rho^3+C)^{4/3}}.\]
\end{ex}

\begin{rmk}
Much of the discussion in this section can be placed in a more general setting such as in higher dimensions or with weaker assumption on integrability conditions. We shall not pursue this kind of generality in this paper. The general theory will be developed elsewhere.
\end{rmk}

Now let us focus on the case where we have a Lagrangian torus fibration. Namely, we shall assume that $(M^\circ,\omega)$ is a symplectic 6-manifold, with $\varphi_0$ an $F$-harmonic $3$-form on $(M^\circ,\omega)$ belonging to the orbit $\clo_0^+$ pointwise so that the associated Lagrangian foliation $\cll$ comes from a Lagrangian torus fibration $\pi:M^\circ\to B$. 

For any point $y\in B$, we shall denote the Lagrangian torus fiber over $y$ by $L_y:=\pi^{-1}(y)$. Inspired by Hitchin \cite{hitchin1997}, we introduce the linear map $l_y:T_yB\to\Omega^1L_y$, mapping a tangent vector in $T_yB$ to a 1-form on $L_y$ as follows. For any $V\in T_yB$ and $p\in L_y$, we define $l_y(V)$ by
\[\left(l_y(V)\right)_p=\iota_{\tilde{V}_p}\omega\bigg|_{T_pL_y},\]
where $\tilde V_p\in T_pM$ is a lifting of $V\in T_yB$ so that $\iota_{\tilde{V}_p}\omega\in T_p^*M$ as a linear functional on $T_pM$ can be restricted to a subspace $T_pL_y$. It is easy to see that this definition is independent of the choice of the lifting as $L_y$ is Lagrangian.

In our case, it is very simple to write down the local expression for $l_y(V)$. As mentioned before, by a theorem of Weinstein \cite{weinstein1971}, we can choose a suitable coordinate chart $\{x^1,y^1,x^2,y^2,x^3,y^3\}$ near $p$ so that
\[\omega=\ud x^1\wedge\ud y^1+\ud x^2\wedge\ud y^2+\ud x^3\wedge\ud y^3\]
and the fiber(leaf) $L_y$ is locally defined by $y=\{y^1,y^2,y^3\}=\textrm{const}$. In other words, the fibration $\pi:M\to B$ is locally given by $(x^1,y^1,x^2,y^2,x^3,y^3)\mapsto(y^1,y^2,y^3)$. Moreover, one can check that $\{x^1,x^2,x^3\}$ is exactly the affine coordinates on $L_y$ with respect to the Bott connection $\nabla^B$. Using such a coordinate chart, we have
\[l_y\left(\frac{\pt}{\pt y^j}\right)=-\ud x^j.\]
For Lagrangian torus fibrations, we know from Theorem \ref{sf} that the induced metrics on $L_y$ are all flat, so the 1-form $-\ud x^j$ is harmonic. Therefore we have proved
\begin{prop}
The linear map $l_y$ induces an isomorphism
\[l_y:T_yB\cong\clh^1(L_y)\]
for every $y\in B$, where $\clh^1(L_y)$ denotes the space of harmonic 1-forms on $L_y$ with respect to the naturally induced metric.
\end{prop}

This proposition allows us to define a Riemannian metric $g_B$ on $B$ as follows: For any tangent vector $X\in T_yB$, we define
\[g_B(X,X)=\int_{L_y}|l_y(X)|^2\vol_{L_y},\]
namely that $g_B$ is induced from the $L^2$-metric on the space of harmonic 1-forms on $L_y$.

Since the torus fibration is locally trivial, over a small open set in $B$, we can fix a basis $\{A_1,A_2,A_3\}$ of $H_1(L,\zz)$, where $L$ is a typical torus fiber. Let 
\[\lambda_{ij}=-\int_{A_i}\ud x^j\]
be the period matrix. As shown in \cite{hitchin1997} we know that the matrix $(\lambda_{ij})$ is invertible and the 1-forms $\xi_i:=\sum_j\lambda_{ij}\ud y^j$ are all closed on $B$ for $i=1,2,3$. It is easy to check that the definition of $\xi_i$ does not depend on the choice of coordinate on $B$, and it differs from itself by a linear transformation if we choose a different basis of $H_1(L,\zz)$.

Since $\xi_i$ are closed, locally on $B$ we can find a function $u_i$ such that $\ud u_i=\xi_i$. And by the above analysis we know that the ambiguity of $u_i$ is up to an affine transformation, therefore we obtain an affine structure on $B$ with affine coordinates $\{u_1,u_2,u_3\}$. Let $\{\alpha^1,\alpha^2,\alpha^3\}$ be the dual basis of $\{A_1,A_2,A_3\}$ in $H^1(L,\zz)$. It follows that the locally defined function $u=u_1\alpha^1+u_2\alpha^2+u_3\alpha^3:B\to H^1(L,\rr)$ does not depend on the choice of $\{A_1,A_2,A_3\}$ and is well-defined up to an additive constant.

\begin{thm}\label{affine}
The $L^2$-metric $g_B$ on $B$ is a Hessian metric with respect to the affine coordinates $\{u_1,u_2,u_3\}$.
\end{thm}
\begin{proof}
We prove this theorem following Hitchin's paper \cite{hitchin1997}. As $H_2(L,\zz)\cong H^1(L,\zz)$ canonically, we denote by $\{B_1,B_2,B_3\}$ the image of $\{\alpha^1,\alpha^2,\alpha^3\}$ in $H_2(L,\zz)$ under this identification. We can form the dual period matrix
\[\mu_{ij}=\int_{B_i}*_y(-\ud x^j),\]
where $*_y$ is the Hodge star operator on $L_y$ with respect to the canonically induced metric. In like manner, we can prove that there exist locally defined functions $v_1,v_2,v_3$ such that
\[\ud v_i=\sum_j\mu_{ij}\ud y^j\]
for $i=1,2,3$. Let $\{\beta^1,\beta^2,\beta^3\}$ be the dual basis of $\{B_1,B_2,B_3\}$ in $H^2(L,\zz)$, then the map
\[v=v_1\beta^1+v_2\beta^2+v_3\beta^3:B\to H^2(L,\rr).\]
It follows that $v$ does not depends on the choice of the basis and is well-defined up to an additive constant. Consider the map $F:B\to H^1(L,\rr)\times H^2(L,\rr)$ given by
\[y\mapsto (u(y),v(y)).\]
The vector space $H^1(L,\rr)\times H^2(L,\rr)$ is equipped with a symplectic form $\Omega$ and a split metric $G$ defined as follows:
\[\begin{split}
\Omega((\alpha^1,\beta^1),(\alpha^2,\beta^2))&:=\int_L\alpha^1\wedge\beta^2-\int_L\alpha^2\wedge\beta^1,\\
G((\alpha^1,\beta^1),(\alpha^1,\beta^1))&:=\int_L\alpha^1\wedge\beta^1.
\end{split}\]
The same argument in \cite{hitchin1997} shows that $F:(B,g_B)\to(H^1(L,\rr)\times H^2(L,\rr),G,\Omega)$ is a locally isometric Lagrangian immersion and that the metric $g_B$ is a Hessian metric with respect to the affine coordinates $\{u_1,u_2,u_3\}$ and its dual $\{v_1,v_2,v_3\}$.
\end{proof}
    
This theorem does not tell us if the local potential for the Hessian metric on $B$ solve a nice real Monge-Amp\`ere equation. In fact, we have the following theorem:
\begin{thm}
The local potentials for the Hessian metric $g_B$ with respect to the affine structures defined in Theorem \ref{affine} solve a real Monge-Amp\`ere equation with constant right hand side if and only if the torus fibers of $\pi:M^\circ\to B$ have the same volume.
\end{thm}
\begin{proof}
From the definition of metric $g_B$ and the periods $\lambda_{ij}$, $\mu_{ij}$, we have that
\[g_B\left(\frac{\p}{\p y^i},\frac{\p}{\p y^j}\right)=\sum_s\lambda_{si}\mu_{sj}.\]
On the other hand, the affine coordinates $u_i$ is defined by
\[\frac{\p u_i}{\p y^j}=\lambda_{ij},\]
so we conclude that
\[g_B\left(\frac{\p}{\p u_i},\frac{\p}{\p u_j}\right)=\mu_{is}\lambda^{sj},\]
hence
\[\det\left[g_B\left(\frac{\p}{\p u_i},\frac{\p}{\p u_j}\right)\right]=\det\mu\cdot\det\lambda^{-1}.\]
A similar calculation indicates that
\[\vol\left(H^1(L_y,\rr/\zz)\right)=\sqrt{\det\mu\cdot\det\lambda^{-1}}.\]
The theorem follows from the fact that the fibers $L_y$ have the same volume if and only if their dual tori $H^1(L_y,\rr/\zz)$ have the same volume.
\end{proof}

Let $\vol(y)$ be the volume of the fiber $L_y$, which can be viewed as a smooth function on $B$. The fibration $\pi:M^\circ\to B$ induces a canonical isomorphism of vector bundles $\ud\pi:TM^\circ/\cll\cong\pi^*TB$ over $M^\circ$. Therefore it is desirable to compare two naturally defined Riemannian metrics $q(\cdot,\cdot):=\omega(\cdot,K\cdot)$ and $\pi^*g_B$ on $TM^\circ/\cll\cong\pi^*TB$.

\begin{thm}\label{GH}
There is an isometry $\ud\pi:(TM^\circ/\cll,q)\cong(\pi^*TB,\pi^*(\vol^{-1}g_B))$.
\end{thm}
\begin{proof}
Previously we have calculated that
\[g_B\left(\frac{\p}{\p y^i},\frac{\p}{\p y^j}\right)=\int_{L_y}g_L(\ud x^i,\ud x^j)\vol_{g_L}=g_L(\ud x^i,\ud x^j)\cdot\vol(y).\]
The last equality holds because $g_L$ is flat hence $g_L(\ud x^i,\ud x^j)$ is a constant on $L_y$. On the other hand, we know that
\[q\left(\frac{\p}{\p y^i},\frac{\p}{\p y^j}\right)=g_L\left(K\frac{\p}{\p y^i},K\frac{\p}{\p y^j}\right)=g_L(\ud x^i,\ud x^j).\]
So we conclude that $(\ud\pi)_*q=\pi^*(\vol^{-1}g_B)$ as desired.
\end{proof}

\section{An Example}

In this section, we investigate a nontrivial example of wonderful degeneration in detail.

Consider an elliptic-fibered K3 surface $\pi:(X,I)\to\ccc\pp^1$, where $I$ is a fixed complex structure on the K3 surface $X$ such that the fibration $\pi$ is holomorphic. Let $\Omega_I=\omega_J+\sqrt{-1}\omega_K$ be a holomorphic volume form on $(X,I)$. By Yau's theorem, we can choose a suitable Ricci-flat K\"ahler metric $\omega_I$ on $X$ such that $(\omega_I,\omega_J,\omega_K)$ form a hyperk\"ahler triple. In particular, we have that $\omega_I^2=\omega_J^2=\omega_K^2$ and $\omega_I\wedge\omega_J=\omega_J\wedge\omega_K=\omega_K\wedge\omega_I=0$.

For every $t>0$, let $\omega_t$ be the unique Ricci-flat K\"ahler metric in the K\"ahler class $\pi^*[\omega_{FS}]+t[\omega_I]$ with respect to the complex structure $I$, where $\omega_{FS}$ is a Fubini-Study metric on $\ccc\pp^1$. Let $c$ be the positive number such that
\[\int_X\omega_I\wedge\pi^*\omega_{FS}=c\cdot\int_X\omega_I^2.\]
It is straightforward to verify that $\left(\dfrac{\omega_t}{\sqrt{t^2+2ct}},\omega_J,\omega_K\right)$ also form a hyperk\"ahler triple in the sense above. Therfore, there exists a complex structure $K_t$ on $X$ such that its holomorphic volume form is given by
\[\Omega_{K_t}=\dfrac{\omega_t}{\sqrt{t^2+2ct}}+\sqrt{-1}\omega_J.\]
Consider the standard torus $(T^2=\rr^2/\zz^2,\ud x\wedge\ud y)$ with the complex structure defined by the 1-form $\sqrt{t^2+2ct}\cdot\ud x+\sqrt{-1}\ud y$. The product symplectic manifold$M:=(X,\omega_K)\times(T^2,\ud x\wedge\ud y)$ is equipped with a holomorphic volume form
\[\Omega_t=\left(\dfrac{\omega_t}{\sqrt{t^2+2ct}}+\sqrt{-1}\omega_J\right)\wedge(\sqrt{t^2+2ct}\cdot\ud x+\sqrt{-1}\ud y)\]
with
\[\varphi_t=\omega_t\wedge\ud x-\omega_J\wedge\ud y.\]
\textbf{Claim}: The family $\{\Omega_t\}_{0<t\leq 1}$ is as a wonderful degeneration of Calabi-Yau structures on $(M,\omega_K+\ud x\wedge\ud y)$.

It is obvious that $\{\Omega_t\}$ is a family of Ricci-flat Calabi-Yau structures by its construction. Therefore we are more concerned with its convergence property. Let $S$ be the union of singular fibers in the elliptic fibration $\pi:X\to\ccc\pp^1$. It was proved by Tosatti \cite{tosatti2010} that there is a smooth convergence $\lim_{t\to 0}\omega_t=\omega_0$ on $X\setminus S$, where $\omega_0$ is the pullback of a K\"ahler metric on $\ccc\pp^1\setminus\pi(S)$. As a consequence, if we set $M^\circ=(X\setminus S)\times T^2$, which is an open and dense subset of $M$, we have the smooth convergence over $M^\circ$:
\[\lim_{t\to 0}\varphi_t=\omega_0\wedge\ud x-\omega_J\wedge\ud y=:\varphi_0.\]
The only thing to check is whether $\varphi_0$ belongs to the orbit $\clo_0^+$ pointwise. Near any point in $X\setminus S$, the fibration $\pi:X\to\ccc\pp^1$ is a holomorphic submersion, hence we can choose a local holomorphic coordinate $(z_1,z_2)$ on $X$ such that 
\[\Omega_I=\omega_J+\sqrt{-1}\omega_K=\ud z_1\wedge\ud z_2\]
and that $\pi$ is given by the second coordinate projection $(z_1,z_2)\mapsto z_2$ where $z_2$ is a local coordinate on $\ccc\pp^1$. If we write $z_j=x_j+\sqrt{-1}y_j$ for $j=1,2$, then we obtain the following local expression for $\varphi_0$:
\[\varphi_0=f\ud x_2\wedge\ud y_2\wedge\ud x-(\ud x_1\wedge\ud x_2-\ud y_1\wedge\ud y_2)\wedge\ud y,\]
where $f$ is a locally defined positive function of $z_2$. Direct calculation shows that
\[\ker K_{\varphi_0}=\spann\left\{\frac{\p}{\p x_1},\frac{\p}{\p y_1},\frac{\p}{\p x}\right\}\]
is 3-dimensional. Hence $\varphi_0$ belongs to the orbit $\clo_0^+$ pointwise. Further computation shows that
\[F(\varphi_0)=4f\ud x_2\wedge\ud y_2\wedge\ud y.\]

Now let us examine the geometry of $(M^\circ,\omega,\varphi_0)$. Using the local coordinates above, we know
\[\omega=\ud x_1\wedge\ud y_2+\ud y_1\wedge\ud x_2+\ud x\wedge\ud y\]
and
\[K\left(\frac{\p}{\p x_2}\right)=-2f\frac{\p}{\p y_1},\quad K\left(\frac{\p}{\p y_2}\right)=-2f\frac{\p}{\p x_1},\quad K\left(\frac{\p}{\p y}\right)=-2\frac{\p}{\p x}.\]
In this case, the Lagrangian foliation $\cll$ comes from a Lagrangian fibration which is the product of $\pi:X\setminus S\to\ccc\pp^1\setminus\pi(S)$ and $T^2=S^1\times S^1\to S^1$, where the second factor is simply the projection to the later copy of $S^1$. The induced Riemannian metric $g_\cll$ is
\[g_\cll=\frac{1}{2f}(\ud x_1\otimes\ud x_1+\ud y_1\otimes\ud y_1)+\frac{1}{2}\ud x\otimes\ud x,\]
which is automatically semi-flat. Direct computation also shows that
\[|\varphi_t|^2=4\sqrt{t^2+2ct}\sim4\sqrt{2c}\cdot t^{1/2}\to 0.\]

Let $q_t$ denote the Riemannian metric $q_t(\cdot,\cdot)=\omega(\cdot,K(\varphi_t)\cdot)$ on $M$. As predicted by Gross-Wilson and Kontsevich-Soibelman-Todorov, Theorem \ref{GH} implies that the Gromov-Hausdorff limit of $(M,q_t)$ is $\ccc\pp^1\times S^1$ equipped with a metric whose regular part is simply $\vol^{-1}\cdot g_B$ on $B=(\ccc\pp^1\setminus\pi(S))\times S^1$. 
 
\bibliographystyle{alpha}

\bibliography{C:/Users/benja/Dropbox/Documents/Source}

\end{document}